\documentclass[oneside,11pt,a4paper,reqno,ps]{article}
\usepackage[a4paper,top=5cm,bottom=5cm,left=3cm,right=3cm,
bindingoffset=5mm]{geometry}
\usepackage[latin1]{inputenc}
\usepackage[english]{babel}
\usepackage[T1]{fontenc}
\usepackage{syntonly}
\usepackage{graphicx}
\usepackage{amssymb}	
\usepackage{amsthm}
\usepackage{microtype}
\usepackage{amsmath}
\usepackage{braket}
\usepackage{eurosym}
\usepackage{amssymb}
\usepackage{amstext}
\usepackage{color}
\usepackage{amsmath,amsthm,amssymb,amscd,epsfig,esint, mathrsfs,graphics,color}
 \usepackage{wrapfig}
\usepackage{verbatim}
\newcommand{\R}{\numberset{R}}

\theoremstyle{plain}
\newtheorem{thm}{Theorem}[section]

\newtheorem{proposition}[thm]{Proposition}
\newtheorem{lemma}[thm]{Lemma}

\theoremstyle{definition}

\def\Xint#1{\mathchoice 
	{\XXint\displaystyle\textstyle{#1}}%
	{\XXint\textstyle\scriptstyle{#1}}%
	{\XXint\scriptstyle\scriptscriptstyle{#1}}%
	{\XXint\scriptscriptstyle\scriptscriptstyle{#1}}%
	\!\int} 
\def\XXint#1#2#3{{\setbox0=\hbox{$#1{#2#3}{\int}$} 
		\vcenter{\hbox{$#2#3$}}\kern-.5\wd0}} 
\def\Mint{\Xint -}

\def\R{\mathbb{R}}
\numberwithin{equation}{section} \makeatletter

\renewcommand{\p@enumi}{\thesection.}
\title{\textbf{Regularity for minimizers of non-autonomous non-quadratic functionals in the case $1<p<2$: \\an a priori estimate}}
\author{Andrea Gentile }

\begin{document}

\maketitle

\begin{abstract}
We establish an a priori estimate for the second derivatives  
of local minimizers of  integral functionals of
the form
$$
\mathcal{F}( v, \Omega )= \int_{\Omega} \! f(x,Dv(x)) \, dx,
$$
with convex integrand  with respect
to the gradient variable, assuming that the function that measures
the oscillation of the integrand with respect to the $x$ variable
belongs to a suitable Sobolev space. \noindent 
The novelty here is that we deal with integrands satisfying subquadratic growth conditions with respect to gradient variable.
\end{abstract}


\bigskip

\bigskip
\maketitle
\markright{\MakeUppercase{Regularity for minimizers}}

\bigskip
\section{Introduction}
\bigskip

	In this paper we consider integral functionals of the  form
	
	\begin{equation}\label{functional}
	\mathcal{F}(v, \Omega)=\int_{\Omega}f(x, Dv(x))dx,
	\end{equation}

	where $\Omega\subset\R^n$ is a bounded open set,  $f:\Omega\times\R^{N\times n}\to\R$ is a Carath\'{e}odory map, such that $\xi\mapsto f(x, \xi)$ is of class $C^2(\R^{N\times n})$, and for an exponent $p\in(1, 2)$ and some constants $L, \alpha, \beta>0$ the following conditions are satisfied:
	
	\begin{equation}\label{f1}
	\frac{1}{L}|\xi|^p\le f(x, \xi)\le L(1+
	|\xi|^p),
	\end{equation}
%
%
	\begin{equation}\label{f2}
	\left<D_\xi f(x, \xi)-D_\xi f(x, \eta), \xi-\eta\right>\ge\alpha\left(1+|\xi|^2+|\eta|^2\right)^\frac{p-2}{2}|\xi-\eta|^2,
	\end{equation}
	
	\begin{equation}\label{f3}
	|D_\xi f(x, \xi)-D_\xi f(x, \eta)|\le\beta\left(1+|\xi|^2+|\eta|^2\right)^\frac{p-2}{2}|\xi-\eta|.
	\end{equation}
	For what concerns the dependence of the energy density on the $x$-variable, we shall assume that the function $D_\xi f(x, \xi)$ is weakly differentiable with respect to $x$ and that $D_x(D_\xi f) \in L^q(\Omega\times\R^{N\times n})$, for some $q>n$.\\
	By the point-wise characterization of the Sobolev functions due to Hajlasz (\cite{H}) this is equivalent to assume that there exists a nonnegative  function $g\in L^q_\text{loc}(\Omega)$ such that
	
	\begin{equation}\label{f4}
	|D_\xi f(x, \xi)-D_\xi f(y, \xi) |\le \left(g(x)+g(y)\right)\left|x-y\right|\left(1+|\xi|^2\right)^\frac{p-1}{2}
	\end{equation} 
	
	for  all $\xi\in\R^{N\times n}$ and for almost every $x,y\in \Omega$.\\
	
	\noindent The regularity properties of minimizers of such integral
functionals   have been widely
investigated in case the energy density $f(x,\xi)$ depends on the
$x$-variable through a continuous function both in the superquadratic and in the subquadratic growth case. In fact, 
it is well known that the partial continuity of the vectorial
minimizers can be obtained with a quantitative modulus of continuity
that depends on the modulus of continuity of the coefficients (see
for example \cite{AF, fh, gm} and the monographs \cite{19,23}
for a more exhaustive treatment). For regularity results under general growth conditions, that of course include the superquadratic and  the subquadratic ones, we refer to \cite{diestrver09,diestrver11}.

Recently, there has been an increasing interest in the study of the
regularity under weaker assumptions on the function that measures
the oscillation of the integrand $f(x,\xi)$ with respect to the
$x$-variable. \\
This study has been successfully carried out when the oscillation
of $f(x,\xi)$ with respect to the $x$-variable is controlled through
a coefficient that belongs to a suitable Sobolev class of integer or fractional order and the assumptions \eqref{f1}--\eqref{f4} are satisfied with an exponent $p\ge 2$.

Actually, it has been shown that the weak
differentiability of the partial map $x\mapsto f(x,\xi)$ transfers to the
gradient of the minimizers of the functional \eqref{functional} (see
\cite{5,EleMarMas, EleMarMas2, GP, 32}) as well as to the gradient  of the
solutions of non linear elliptic systems (see \cite{6,CGP,CuMR,33}) and of non linear systems with degenerate ellipticity (see \cite{Gio}).

 As far as we know, no higher differentiability results are available for vectorial minimizers under the so-called subquadratic growth conditions, i.e. when the  assumptions \eqref{f1}--\eqref{f4} hold true for an exponent $1<p\le 2$ in case of Sobolev coefficients.
The aim of this paper is to start the study of the higher differentiability properties of local minimizers of integral functional \eqref{functional} under subquadratic growth condition. More precisely, we shall establish the following
a priori estimate for the second derivatives of the local minimizers.

\begin{thm}\label{main}
Let $u\in W^{2,p}_{\mathrm{loc}}(\Omega;\R^N)$ be a local minimizer of the functional 	$\mathcal{F}(v, \Omega)$ under the assumptions \eqref{f1}--\eqref{f4}. If $q\ge\frac{2n}{p}$, than the following estimate

\begin{equation}\label{mainestimate}
\Arrowvert D^2u\Arrowvert_{L^{p}(B_{r})}\le C(\alpha, \beta, p, n) \left(\Arrowvert Du\Arrowvert_{L^{p}(B_{R})}+\Arrowvert g\Arrowvert_{L^{q}(B_{R})}\right)
\end{equation}

holds true for every $0<r<R$ such that $B_R\Subset\Omega$ with $C=C(\alpha, \beta, p, n).$
\end{thm}
 
The main tool to establish previous result is the use of the so called difference quotient method and a double iteration to reabsorb terms with critical summability. Respect to previous papers on this subject,  new technical difficulties arise since we  are dealing with the subquadratic growth case.

	\section{Preliminary results}
	\markright{\MakeUppercase{Regularity for minimizers}}
	In this section we shall collect some results that will be useful to achieve our main result.\\
	In this section we recall some standard definitions and collect
several lemmas that we shall need to establish our  results. We
shall follow the usual convention and denote by $C$ or $c$ a general
constant that may vary on different occasions, even within the same
line of estimates. Relevant dependencies on parameters and special
constants will be suitably emphasized using parentheses or
subscripts. All the norms we use on $\R^n$, $\R^N$ and $\R^{n\times N}$ will be
the standard Euclidean ones and denoted by $| \cdot |$ in all cases.
In particular, for matrices $\xi$, $\eta \in \R^{n\times N}$ we write $\langle
\xi, \eta \rangle : = \text{trace} (\xi^T \eta)$ for the usual inner
product of $\xi$ and $\eta$, and $| \xi | : = \langle \xi, \xi
\rangle^{\frac{1}{2}}$ for the corresponding Euclidean norm. When $a
\in \R^N$ and $b \in \R^n$ we write $a \otimes b \in \R^{n\times N}$ for the
tensor product defined
as the matrix that has the element $a_{r}b_{s}$ in its r-th row and s-th column.\\
For  a  $C^2$ function $f \colon \Omega\times\R^{n\times N} \to \R$, we write
$$
D_\xi f(x,\xi )[\eta ] := \frac{\rm d}{{\rm d}t}\Big|_{t=0} f(x,\xi
+t\eta )\quad \mbox{ and } \quad D_{\xi\xi}f(x,\xi )[\eta ,\eta ] :=
\frac{\rm d^2}{{\rm d}t^{2}}\Big|_{t=0} f(x,\xi +t\eta )
$$
for $\xi$, $\eta \in \R^{n\times N}$ and for almost every $x\in \Omega$.\\
 With the symbol $B(x,r)=B_r(x)=\{y\in
\R^n:\,\, |y-x|<r\}$, we will denote the ball centered at $x$ of
radius $r$ and
$$(u)_{x_0,r}= \Mint_{B_r(x_0)}u(x)\,dx,$$
stands for the integral mean of $u$ over the ball $B_r(x_0)$. We
shall omit the dependence on the center  when it is clear from the context.
		
	\subsection{An auxiliary function}
	\markright{\MakeUppercase{Regularity for minimizers}}
	As usual, we shall use the following auxiliary function
	
	\begin{equation}
	V_p(\xi):=\left(1+|\xi|^2\right)^\frac{p-2}{4}\xi, \text{ for all } \xi\in \R^{N\times n}.
	\end{equation}

	
	\subsection{Some useful lemmas}
	
	The following result is proved in \cite{AF}, and will be useful to estimate the $L^p$ norm of $D^2u$, using the $L^2$ norm of the gradient of $V_p(Du).$
	
	\begin{lemma}\label{lemma2.2AF}
		For every $\gamma\in\left(-\frac{1}{2}, 0\right)$ and $\mu\ge0$ we have
		\begin{equation}
		(2\gamma+1)|\xi-\eta|\le\frac{|(\mu^2+|\xi|^2)^\gamma\xi-(\mu^2+|\eta|^2)^\gamma\eta}{(\mu^2+|\xi|^2+|\eta|^2)^\gamma}\le\frac{c(k)}{2\gamma+1}|\xi-\eta|,
		\end{equation}
		for every $\xi, \eta \in\R^k.$
	\end{lemma}
	
	\begin{lemma}\label{lemma2.3AF}
		For every $\gamma\in\left(-\frac{1}{2}, 0\right)$  we have
		\begin{equation}
		c_0(\gamma)(1+|\xi|^2+|\eta|^2)^\gamma\le\int_0^1(1+|t\xi+(1-t)\eta|^2)^\gamma\,dt\le c_1(\gamma)(1+|\xi|^2+|\eta|^2)^\gamma,
		\end{equation}
		for every $\xi, \eta \in\R^k.$
	\end{lemma}
	
	The next lemma can be proved using an iteration technique, and will be very useful in the following, where we will refer to this as \text{Iteration Lemma}.
	
	\begin{lemma}[Iteration Lemma]\label{iteration}
		Let $h: [\rho, R]\to \R$ be a nonnegative bounded function, $0<\theta<1$, $A, B\ge0$ and $\gamma>0$. Assume that 
		
		$$
		h(r)\le\theta h(d)+\frac{A}{(d-r)^\gamma}+B
		$$
		
		for all $\rho\le r<d\le R_0.$ Then
		
		$$
		h(\rho)\le \frac{c(A)}{(R_0-\rho)^\gamma}+cB,
		$$
		
		where $c=c(\theta, \gamma)>0$.
	\end{lemma}
	For the proof we refer to \cite[Lemma 6.1]{23}.

	\subsection{Finite difference and difference quotient}
	\markright{\MakeUppercase{Regularity for minimizers}}
	In what follows, we denote, for every function $f$,  $h\in \R$, and being $e_s$ the unit vector in the $x_s$ direction,

	$$
	\tau_{s, h}f(x):=f(x+he_s)-f(x)
	$$
	
	defines the finite difference operator.
	
	Here we  recall some  properties of the finite difference, that will be useful in the following.
	
	\begin{proposition}\label{diffquot}
		Let $f$ and $g$ be two functions such that $f, g \in W^{1,p}(\Omega, \R^n)$ with $p\ge1$, and let us consider the set
		
		$$
		\Omega_{|h|}:=\{x \in{\Omega}: \text{dist}(x, \partial\Omega)>|h|\}.
		$$
		
		Then the following properties hold:
		\begin{enumerate}
			\item\label{1} $\tau_{s, h}f\in W^{1,p}(\Omega_{|h|})$ and 
			$$
			D_i(\tau_{s, h}f)=\tau_{s, h}(D_if);
			$$
			\item\label{2} if at least one of the functions $f$ or $g$ has support contained in $\Omega_{|h|}$, then
			$$
			\int_{\Omega}f\tau_{s, h}gdx=\int_{\Omega}g\tau_{s, -h}fdx;
			$$
			\item\label{3} we have $$\tau_{s, h}(fg)(x)=f(x+he_s)\tau_{s, h}g(x)+g(x)\tau_{s, h}f(x).$$
		\end{enumerate}
	\end{proposition}

	The following lemmas describe fundamental properties of finite differences and difference quotients of Sobolev functions.
	
	\begin{lemma}\label{lem1}
		If $0<\rho<R$, $|h|<\frac{R-\rho}{2}$, $1<p<+\infty$, $s\in\{1,..., n\}$ and $f, D_sf \in L^p(B_R)$, then
		
		$$
		\int_{B_\rho}|\tau_{s,h}f(x)|^pdx\le|h|^p\int_{B_R}|D_sf(x)|^pdx.
		$$
		
		Moreover, for $\rho<R$, $|h|<\frac{R-\rho}{2}$,
		
		$$
		\int_{B_\rho}|f(x+he_s)|^pdx\le c(n, p)\int_{B_R}|f(x)|^pdx.
		$$
	\end{lemma}

	\begin{lemma}
		Let $f:\R^n\to\R^N$, $f\in L^p(B_R)$ with $1<p<+\infty$. Suppose that there exist $\rho\in(0, R)$ and $M>0$ such that
		
		$$
		\sum_{s=1}^{n}\int_{B_\rho}|\tau_{s, h}f(x)|^pdx\le M^p|h|^p
		$$
			
		for every $h<\frac{R-\rho}{s}$. Then $f\in W^{1,p}(B_R, \R^N)$. Moreover
		
		$$
		\Arrowvert Df \Arrowvert_{L^p(B_\rho)}\le M.
		$$		
	\end{lemma}
	
\section{Proof of Theorem \ref{main}}
\markright{\MakeUppercase{Regularity for minimizers}}

		It is well known that every local minimizer of the functional \eqref{functional} is a weak solution $u\in W^{1,p}(\Omega, \R^N)$ of the corresponding Euler-Lagrange system, i.e.
	
%
%
%
	
	\begin{equation}\label{diveq}
	\text{div}A(x, Du(x))=0,
	\end{equation}
	where we set 
	\begin{equation}\label{Adef}
	A_i^\alpha(x, \xi):=D_{\xi^\alpha_i}f(x, \xi), \text{ for all } \alpha=1,..., N\,\,\text{and}\,\, i=1,\dots,n.
	\end{equation}
	
Assumptions \eqref{f1}, \eqref{f2}, \eqref{f3}, can be written as
	
	
%
	
	\begin{equation}\label{A1}
	\left<A(x, \xi)-A(x, \eta), \xi-\eta\right>\ge\alpha|\xi-\eta|^2\left(1+|\xi|^2+|\eta|^2\right)^\frac{p-2}{2},
	\end{equation}

	\begin{equation}\label{A2}
	|A(x, \xi)-A(x, \eta)|\le\beta|\xi-\eta|\left(1+|\xi|^2+|\eta|^2\right)^\frac{p-2}{2}
	\end{equation}
for every $\xi$, $\eta \in \R^{n\times N}$ and for almost every $x \in \Omega$.
	
	Concerning the dependence on the $x$-variable, assumption \eqref{f4} translates into the following
	
	\begin{equation}\label{A3}
	|A(x, \xi)-A(y, \xi)|\le \left(g(x)+g(y)\right)\left|x-y\right|\left(1+|\xi|^2\right)^\frac{p-1}{2}
	\end{equation} 
	
	for every $\xi$, $\eta \in \R^{N\times n}$ and for almost every $x, y\in \Omega$.

%

\begin{proof}[Proof of Theorem \ref{main}]	Let us fix a ball $B_R(x_0)=B_R$ of radius $R\in(0, \text{dist}(x_0, \partial\Omega))$, and consider $\frac{R}{2}<r<\tilde{s}<t<\lambda r<R<1$, with $1<\lambda<2$.  Let's test the equation \eqref{diveq} with the function $\varphi=\tau_{s, -h}(\eta^2\tau_{s, h}u)$, where $\eta\in C^\infty_0(B_t)$ is a cut off function such that $\eta=1$ on $B_{\tilde{s}}$, $|D\eta|\le\frac{c}{t-\tilde{s}}.$\\
	With this choice of $\varphi$, and by  \ref{2} of Proposition \ref{diffquot}, we get
	
	$$
	\int_{B_R}\left<\tau_{s, h}A(x, Du(x)), D(\eta^2(x)(\tau_{s, h}u(x)))\right>dx=0.
	$$
	
	After some manipulations, and dropping the vector $e_s$ to simplify the notations, we can write the last equivalence as follows
	
	\begin{equation*}
	\begin{split}
		&I_0:=\int_{B_R}\left<A(x+h, Du(x+h))-A(x+h, Du(x)), \eta^2(x)D(\tau_{s, h}u(x))\right>dx\\=&-\int_{B_R}\left<A(x+h, Du(x))-A(x, Du(x)), \eta^2(x)D(\tau_{s,h}u(x))\right>dx\\&-\int_{B_R}\left<\tau_{s, h}A(x, Du(x)), 2\eta(x) D\eta(x)\otimes\tau_{s, h}u(x)\right>\\
		&=-\int_{B_R}\left<A(x+h, Du(x))-A(x, Du(x)), \eta^2(x)D(\tau_{s,h}u(x))\right>dx\\&-\int_{B_R}\left<A(x, Du(x)), \tau_{s, -h}\Big(2\eta(x) D\eta(x)\otimes\tau_{s, h}u(x)\Big)\right>\\
		& =-\int_{B_R}\left<A(x+h, Du(x))-A(x, Du(x)), \eta^2(x)D(\tau_{s,h}u(x))\right>dx
		\\&-\int_{B_R}\left<A(x, Du(x)), \tau_{s, -h}\Big(2\eta(x) D\eta(x)\Big)\otimes\tau_{s, h}u(x)\right>dx
		\\&-\int_{B_R}\left<A(x, Du(x)), 2\eta(x) D\eta(x)\otimes\tau_{s, -h}\Big(\tau_{s, h}u(x)\Big)\right>dx
		:I+II+III.
	\end{split}
	\end{equation*}
	Previous equality implies that
	\begin{equation}\label{main1}
	I_0\le |I|+|II|+|III|.
	\end{equation}
	
%

%
	
	In order to estimate the integral $|I|$, we use   the hypothesis \eqref{A3}  and  Young's inequality, as follows
	
	\begin{equation}\label{I}
	\begin{split}
	|I|&\le c|h|\int_{B_R}\eta^2(x)\left(g(x)+g(x+h)\right)\left(1+|Du(x)|^2\right)^{\frac{p-1}{2}}|D\tau_{s, h}u(x)|dx \\&\le c|h|\int_{B_R}\eta^2(x)\left(g(x)+g(x+h)\right)\left(1+|Du(x)|^2+|Du(x+h)|^2\right)^{\frac{p-1}{2}}|D(\tau_{s, h}u(x))|dx\\&
	\le\varepsilon\int_{B_R}\eta^2(x)|D(\tau_{s, h}u(x))|^2\left(1+|Du(x)|^2+|Du(x+h)|^2\right)^{\frac{p-2}{2}}dx\\&+
	c_\varepsilon|h|^2\int_{B_R}\eta^2(x)\left(g^2(x)+g^2(x+h)\right)\left(1+|Du(x)|^2+|Du(x+h)|^2\right)^\frac{p}{2}dx.
	\end{split}
	\end{equation}
Now, we  estimate $|II|$ by  \eqref{A2} and the properties of $\eta$ thus obtaining
	
	\begin{equation*}\label{II}
	\begin{split}
	|II|&\le \frac{c|h|}{(t-\tilde{s})^2}\int_{B_t}\left(1+|Du(x)|^2\right)^\frac{p-1}{2}|\tau_{s, h}u(x)|dx\\
	&\le \frac{c|h|}{(t-\tilde{s})^2}\left(\int_{B_t}\left(1+|Du(x)|^2\right)^\frac{p}{2}\right)^{\frac{p-1}{p}}\left(\int_{B_t}|\tau_{s, h}u(x)|^pdx\right)^{\frac{1}{p}} ,
	\end{split}
	\end{equation*}
	where, in the last inequality, we used H\"older's inequality. By virtue of Lemma \ref{lem1}, we obtain
	\begin{equation}\label{II}
	|II|\le \frac{c|h|^2}{(t-\tilde{s})^2}\int_{B_{\lambda r}}\left(1+|Du(x)|^2\right)^\frac{p}{2}.
	\end{equation}
	The term $|III|$ is estimated using the hypothesis \eqref{A2}, the properties of $\eta$, H\"older's inequality and Lemma \ref{lem1}, as follows
	
	\begin{equation}\label{III}
	\begin{split}
	|III|&\le \frac{c}{t-\tilde{s}}\int_{B_t}\left(1+|Du(x)|^2\right)^{\frac{p-1}{2}}|\tau_{s, -h}(\tau_{s, h}u(x))|dx\\
	&\le \frac{c}{t-\tilde{s}}\left(\int_{B_t}\left(1+|Du(x)|^2\right)^{\frac{p}{2}}\right)^{\frac{p-1}{p}}\left(\int_{B_t}|\tau_{s, -h}(\tau_{s, h}u(x))|^pdx\right)^{\frac{1}{p}}
	\\
	&\le \frac{c|h|}{t-\tilde{s}}\left(\int_{B_t}\left(1+|Du(x)|^2\right)^{\frac{p}{2}}\right)^{\frac{p-1}{p}}\left(\int_{B_{\lambda r}}|\tau_{s, h}Du(x))|^pdx\right)^{\frac{1}{p}},
	\end{split},
	\end{equation}
	where in the last inequality we used Lemma \ref{lem1} and \eqref{1} of Proposition \ref{diffquot}.
	By the assumption \eqref{A1}, we get
	
	\begin{equation}\label{I_0}
	|I_0|\ge c(p, \alpha)\int_{B_R}\eta^2(x)\left(1+|Du(x)|^2+|Du(x+h)|^2\right)^\frac{p-2}{2}|\tau_{s, h}Du(x)|^2dx.
	\end{equation}
	
	 Inserting estimates  \eqref{I}, \eqref{II}, \eqref{III} and \eqref{I_0} in \eqref{main1}, we obtain
	
	\begin{equation}
	\begin{split}
	&c(p, \alpha)\int_{B_R}\eta^2(x)\left(1+|Du(x)|^2+|Du(x+h)|^2\right)^\frac{p-2}{2}|\tau_{s, h}Du(x)|^2dx\\&\le\varepsilon\int_{B_R}\eta^2(x)|D(\tau_{s, h}u(x))|^2\left(1+|Du(x)|^2+|Du(x+h)|^2\right)^{\frac{p-2}{2}}dx\\&
	+c_\varepsilon|h|^2\int_{B_R}\eta^2(x)\left(g^2(x)+g^2(x+h)\right)\left(1+|Du(x)|^2+|Du(x+h)|^2\right)^\frac{p}{2}dx
	\\&+\frac{c|h|^2}{(t-\tilde{s})^2}\int_{B_{\lambda r}}\left(1+|Du(x)|^2\right)^\frac{p}{2}
	\\&+\frac{c|h|}{t-\tilde{s}}\left(\int_{B_t}\left(1+|Du(x)|^2\right)^{\frac{p}{2}}\right)^{\frac{p-1}{p}}\left(\int_{B_{\lambda r}}|\tau_{s, h}Du(x))|^pdx\right)^{\frac{1}{p}}.
	\end{split}
	\end{equation} 
	Choosing $\varepsilon=\frac{c(p, \alpha)}{2}$ in previous estimate, we can reabsorb the first integral in the right hand side by the left hand side thus getting
	
	\begin{equation}
	\begin{split}
	&\int_{B_R}\eta^2(x)\left(1+|Du(x)|^2+|Du(x+h)|^2\right)^\frac{p-2}{2}|\tau_{s, h}Du(x)|^2dx\\&\le
	c|h|^2\int_{B_R}\eta^2(x)\left(g^2(x)+g^2(x+h)\right)\left(1+|Du(x)|^2+|Du(x+h)|^2\right)^\frac{p}{2}dx
	\\&+\frac{c|h|^2}{(t-\tilde{s})^2}\int_{B_{\lambda r}}\left(1+|Du(x)|^2\right)^\frac{p}{2}
	\\&\frac{c|h|}{t-\tilde{s}}\left(\int_{B_t}\left(1+|Du(x)|^2\right)^{\frac{p}{2}}\right)^{\frac{p-1}{p}}\left(\int_{B_{\lambda r}}|\tau_{s, h}Du(x))|^pdx\right)^{\frac{1}{p}},
	\end{split}
	\end{equation} 
	with $c=c(\alpha,\beta,p,n)$.
	Dividing previous estimate  by $|h|^2$ and using Lemma \ref{lemma2.2AF}, we have
	
	\begin{equation}\label{24}
	\begin{split}&\int_{B_R}\eta^2(x)\frac{|\tau_{s, h}(V_p(Du))|^2}{|h|^2}\\&\le
	c\int_{B_R}\eta^2(x)\left(1+|Du(x)|^2+|Du(x+h)|^2\right)^\frac{p-2}{2}\frac{|\tau_{s, h}Du(x)|^2}{|h|^2}dx\\&
	\le
	c\int_{B_R}\eta^2(x)\left(g^2(x)+g^2(x+h)\right)\left(1+|Du(x)|^2+|Du(x+h)|^2\right)^\frac{p}{2}dx
	\\&+\frac{c}{(t-\tilde{s})^2}\int_{B_{\lambda r}}\left(1+|Du(x)|^2\right)^\frac{p}{2}
	\\&+\frac{c}{t-\tilde{s}}\left(\int_{B_t}\left(1+|Du(x)|^2\right)^{\frac{p}{2}}\right)^{\frac{p-1}{p}}\left(\int_{B_{\lambda r}}\frac{|\tau_{s, h}Du(x))|^p}{|h|^p}dx\right)^{\frac{1}{p}},
	\end{split}
	\end{equation} 
	 
	 Now, by H\"{o}lder's inequality and Lemma \ref{lemma2.2AF}, we get 
	 
	 \begin{equation}\label{25}
	 \begin{split}
	 &\int_{B_R}\eta^2(x)\frac{\left|\tau_{s, h}Du(x)\right|^p}{|h|^p}dx\\
	 &\le\int_{B_R}\eta^2(x)\frac{\left|\tau_{s, h}\left(V_p\left(Du\right)\right)\right|^p}{|h|^p}\left(1+\left|Du(x)\right|^2+\left|Du(x+h)\right|^2\right)^\frac{p(2-p)}{4}\\&\le\left(\int_{B_R}\eta^2(x)\frac{\left|\tau_{s, h}\left(V_p\left(Du\right)\right)\right|^2}{|h|^2}\right)^\frac{p}{2}\left(\int_{B_R}\eta^2(x)\left(1+\left|Du(x)\right|^2+\left|Du(x+h)\right|^2\right)^\frac{p}{2}\right)^\frac{2-p}{2}dx,
	 \end{split}
	 \end{equation} 
	 and therefore, combining \eqref{24} and \eqref{25}, we have
	 
	 \begin{equation}\label{26}
	 \begin{split}
	 &\int_{B_R}\eta^2(x)\frac{\left|\tau_{s, h}Du(x)\right|^p}{|h|^p}dx\\&\le c\Bigg\{
	\int_{B_R}\eta^2(x)\left(g^2(x)+g^2(x+h)\right)\left(1+|Du(x)|^2+|Du(x+h)|^2\right)^\frac{p}{2}dx
	\\&+\frac{c}{(t-\tilde{s})^2}\int_{B_{\lambda r}}\left(1+|Du(x)|^2\right)^\frac{p}{2}
	\\&+\frac{c}{t-\tilde{s}}\left(\int_{B_t}\left(1+|Du(x)|^2\right)^{\frac{p}{2}}\right)^{\frac{p-1}{p}}\left(\int_{B_{\lambda r}}\frac{|\tau_{s, h}Du(x))|^p}{|h|^p}dx\right)^{\frac{1}{p}}\Bigg\}^\frac{p}{2}\\&\cdot\left\{\int_{B_R}\eta^2(x)\left(1+\left|Du(x)\right|^2+\left|Du(x+h)\right|^2\right)^\frac{p}{2}\right\}^\frac{2-p}{2}dx.
	 \end{split}
	 \end{equation}
	 
%
 
Using Young's inequality with exponents $\frac{2}{p}$ and $\frac{2}{2-p}$ and the properties of $\eta$, we have

\begin{equation}\label{26}
	 \begin{split}
	 &\int_{B_R}\eta^2(x)\frac{\left|\tau_{s, h}Du(x)\right|^p}{|h|^p}dx\\&\le 
	c\int_{B_R}\eta^2(x)\left(g^2(x)+g^2(x+h)\right)\left(1+|Du(x)|^2+|Du(x+h)|^2\right)^\frac{p}{2}dx
	\\&+\left(1+\frac{c}{(t-\tilde{s})^2}\right)\int_{B_{\lambda r}}\left(1+|Du(x)|^2\right)^\frac{p}{2}
	\\&+\frac{c}{t-\tilde{s}}\left(\int_{B_t}\left(1+|Du(x)|^2\right)^{\frac{p}{2}}\right)^{\frac{p-1}{p}}\left(\int_{B_{\lambda r}}\frac{|\tau_{s, h}Du(x))|^p}{|h|^p}dx\right)^{\frac{1}{p}}.
	 \end{split}
	 \end{equation}

Using Young's inequality with exponents $p$ and $\frac{p}{p-1}$ to estimate the last integral in the left side, we obtain

\begin{equation}\label{29}
\begin{split}
&\int_{B_R}\eta^2(x)\frac{\left|\tau_{s, h}Du(x)\right|^p}{|h|^p}dx \le c\int_{B_{\lambda r}}g^2(x)dx+c\int_{B_{\lambda r}}g^2(x)\left|Du(x)\right|^pdx\\&+c\left(1+\frac{1}{(t-\tilde{s})^2}+\frac{1}{(t-\tilde{s})^{\frac{p}{p-1}}}\right)\int_{B_R}\left(1+\left|Du(x)\right|^2\right)^\frac{p}{2}dx\\&+\frac{1}{2}\int_{B_{\lambda r} }\frac{\left|\tau_{s, h}Du(x)\right|^p}{|h|^p}dx.
\end{split}
\end{equation}

Recalling the properties of $\eta$, we obtain

\begin{equation}\label{30}
\begin{split}
\int_{B_{\tilde{s}}}\frac{\left|\tau_{s, h}Du(x)\right|^p}{|h|^p}dx &\le \frac{1}{2}\int_{B_{\lambda r} }\frac{\left|\tau_{s, h}Du(x)\right|^p}{|h|^p}dx\\
&+c\int_{B_{\lambda r}}g^2(x)dx+c\int_{B_{\lambda r}}g^2(x)\left|Du(x)\right|^pdx\\&+c\left(1+\frac{1}{(t-\tilde{s})^2}+\frac{1}{(t-\tilde{s})^{\frac{p}{p-1}}}\right)\int_{B_R}\left(1+\left|Du(x)\right|^2\right)^\frac{p}{2}dx.
\end{split}
\end{equation}

Since the previous estimate holds for every $r<\tilde{s}<t<\lambda r$, the Lemma \ref{iteration} implies

\begin{equation}\label{31}
\begin{split}
\int_{B_{r}}\frac{\left|\tau_{s, h}Du(x)\right|^p}{|h|^p}dx&\le c\int_{B_{\lambda r}}g^2(x)dx+c\int_{B_{\lambda r}}g^2(x)\left|Du(x)\right|^pdx\\&+c\left(1+\frac{1}{r^2(\lambda-1)^2}+\frac{1}{r^\frac{p}{p-1}(\lambda-1)^\frac{p}{p-1}}\right)\int_{B_{\lambda r}}\left(1+\left|Du(x)\right|^2\right)^\frac{p}{2}dx
\end{split}
\end{equation}
and so, by Lemma \ref{lem1},
\begin{equation}\label{31}
\begin{split}
\int_{B_{r}}|D^2u|^pdx&\le c\int_{B_{\lambda r}}g^2(x)dx+c\int_{B_{\lambda r}}g^2(x)\left|Du(x)\right|^pdx\\&+c\left(1+\frac{1}{r^2(\lambda-1)^2}+\frac{1}{r^\frac{p}{p-1}(\lambda-1)^\frac{p}{p-1}}\right)\int_{B_{\lambda r}}\left(1+\left|Du(x)\right|^2\right)^\frac{p}{2}dx.
\end{split}
\end{equation}
To go further in the estimate, we have to study the term

\begin{equation}\label{32}
\int_{B_{\lambda r}}g^2(x)\left|Du(x)\right|^pdx,
\end{equation}
and to do this, our first step is to apply H\"{o}lder's inequality with exponents $\frac{q}{2}$ and $\frac{q}{q-2}$, thus obtaining\\

\begin{equation}\label{33}
\int_{B_{\lambda r}}g^2(x)\left|Du(x)\right|^pdx\le\left(\int_{B_{\lambda r}}g^q(x)dx\right)^\frac{2}{q}\left(\int_{B_{\lambda r}}\left|Du(x)\right|^\frac{pq}{q-2}dx\right)^\frac{q-2}{q}.
\end{equation}

Now we observe that, by Sobolev's embedding Theorem, if $u\in W^{2,p}_{\text{loc}}(\Omega)$, then $Du\in L^{q'}_{\text{loc}}(\Omega)$ for all $q'\in[p, p^*]$, where $p^*=\frac{np}{n-p}$. So, the second integral in the right hand side of \eqref{33}, converges for $\frac{pq}{q-2}\le\frac{np}{n-p}$, that is $q\ge\frac{2n}{p}$.\\
We have to distinguish between two cases.

\textbf{Case I. $\frac{pq}{q-2}=\frac{np}{n-p}$.}\\

In this case we have $q=\frac{2n}{p}$, then, by Sobolev's inequality,

\begin{equation}\label{34}
\begin{split}
\left(\int_{B_{\lambda r}}g^q(x)dx\right)^\frac{2}{q}&\left(\int_{B_{\lambda r}}\left|Du(x)\right|^\frac{pq}{q-2}dx\right)^\frac{q-2}{q}=\left(\int_{B_{\lambda r}}g^q(x)dx\right)^\frac{2}{q}\left(\int_{B_{\lambda r}}\left|Du(x)\right|^\frac{np}{n-p}\right)^\frac{n-p}{n}\\&\le c \left(\int_{B_{\lambda r}}g^q(x)dx\right)^\frac{2}{q}\int_{B_{\lambda r}}(\left|D^2u\right|^p+|Du|^p)dx.
\end{split}
\end{equation}

By the absolute continuity if the integral, there exists $R_0>0$ such that, for every $R<R_0$, we have

\begin{equation}\label{35}
c\left(\int_{B_{R}}g^q(x)dx\right)^\frac{2}{q}<\frac{1}{2}.
\end{equation}

For this choice of $R$, joining \eqref{31}, \eqref{33}, \eqref{34}, \eqref{35}, we get:

\begin{equation}\label{36}
\begin{split}
\int_{B_{r}}\left|D^2u(x)\right|^pdx&\le c\int_{B_{\lambda r}}g^2(x)dx+\frac{1}{2}\int_{B_{\lambda r}}\left|D^2u(x)\right|^pdx\\&+\left(c+\frac{c}{r^2(\lambda-1)^2}+\frac{c}{r^\frac{p}{p-1}(\lambda-1)^\frac{p}{p-1}}\right)\int_{B_{\lambda r}}\left(1+\left|Du(x)\right|^2\right)^\frac{p}{2}dx.
\end{split}
\end{equation}

\textbf{Case II. $\frac{pq}{q-2}<\frac{np}{n-p}.$}\\
In this case we have $q>\frac{2n}{p}$.\\
Since $u\in W^{2,p}_{\text{loc}}$, then $Du\in W^{1,p}_{\text{loc}}(\Omega)$ and $D^2u\in L^p_{\text{loc}}(\Omega)$. Recalling that, by Sobolev's embedding Theorem, we have $W^{1,p}_{\text{loc}}(\Omega)\hookrightarrow L^{q'}_{\text{loc}}(\Omega)$ for all $q'\in[p, p^*]$, where $p^*=\frac{np}{n-p}$, we have, for a constant $c=c(n, p)$,

\begin{equation}\label{imm1}
\Arrowvert Du\Arrowvert_{L^{q'}(B_{\lambda r})}\le c\Arrowvert Du \Arrowvert_{W^{1, p}(B_{\lambda r})}\le c\left(\Arrowvert Du \Arrowvert_{L^{p}(B_{\lambda r})}+\Arrowvert D^2u \Arrowvert_{L^{p}(B_{\lambda r})}\right).
\end{equation}

Now since, for $q'=\frac{pq}{q-2}$ we have $p<q'<p^*$, then $L^{q'}_{\text{loc}}(\Omega)\hookrightarrow L^{p}_{\text{loc}}(\Omega)$, then

\begin{equation}\label{imm2}
\Arrowvert Du \Arrowvert_{L^{p}(B_{\lambda r})}\le c \Arrowvert Du \Arrowvert_{L^{q'}(B_{\lambda r})}.
\end{equation}

Joining \eqref{imm1} and \eqref{imm2}, we get

\begin{equation}
\left(\int_{B_{\lambda r}}\left|Du(x)\right|^\frac{pq}{q-2}dx\right)^\frac{q-2}{pq}\le c\left(\int_{B_{\lambda r}}(\left|D^2u\right|^p+|Du|^p)dx\right)^\frac{1}{p}
\end{equation}

that is

\begin{equation*}
\left(\int_{B_{\lambda r}}\left|Du(x)\right|^\frac{pq}{q-2}dx\right)^\frac{q-2}{q}\le c\left(\int_{B_{\lambda r}}(\left|D^2u\right|^p+|Du|^p)dx\right).
\end{equation*}

So we obtain

\begin{equation*}
\begin{split}
&\left(\int_{B_{\lambda r}}g^q(x)dx\right)^\frac{2}{q}\left(\int_{B_{\lambda r}}\left|Du(x)\right|^\frac{pq}{q-2}dx\right)^\frac{q-2}{q}\\
&\le c\left(\int_{B_{\lambda r}}g^q(x)dx\right)^\frac{2}{q}\left(\int_{B_{\lambda r}}(\left|D^2u\right|^p+|Du|^p)dx\right)	
\end{split}
\end{equation*}

and by the absolute continuity of the integral, as in the previous case, choosing the value of $r$ opportunely, we get an estimate like \eqref{36} in this case too.

Since \eqref{36} holds for all $r$ and for all $\lambda\in(1, 2)$, setting $\rho=r$, $R_0=\lambda r$, $\gamma=\frac{p}{p-1}$ and

\begin{equation*}
h(\rho)=\int_{B_\rho}\left|D^2u(x)\right|dx,
\end{equation*}

 by Lemma \ref{iteration}, we have 

\begin{equation}\label{estimate1}
\Arrowvert D^2u\Arrowvert_{L^{p}(B_{r})}\le c(\alpha, \beta, p, n) \left(\Arrowvert Du\Arrowvert_{L^{p}(B_{\lambda r})}+\Arrowvert g\Arrowvert_{L^{2}(B_{\lambda r})}\right).
\end{equation}

Since $q\ge\frac{2n}{p}>2$, we have $L^q_{\text{loc}}(\Omega)\hookrightarrow L^2_{\text{loc}}(\Omega)$, and by \eqref{estimate1} we get

\begin{equation}\label{estimate2}
\Arrowvert D^2u\Arrowvert_{L^{p}(B_{r})}\le C(\alpha, \beta, p, n) \left(\Arrowvert Du\Arrowvert_{L^{p}(B_{\lambda r})}+\Arrowvert g\Arrowvert_{L^{q}(B_{\lambda r})}\right),
\end{equation}

that is \eqref{mainestimate}.
\end{proof}

	{Andrea Gentile
	\\
	Dipartimento di Matematica e Applicazioni ``R. Caccioppoli''
\\
Universit\`a degli Studi di Napoli ``Federico II''\\
Via Cintia, 80126, Napoli (Italy)
\\andr.gentile@studenti.unina.it}
\end{document}